\documentclass[12pt]{amsart}
\usepackage{amsfonts, amsmath, amsthm, amssymb}

\usepackage{graphicx}
\usepackage{epstopdf}
\usepackage{psfrag}

\newtheorem{pro}{Proposition}[section]
\newtheorem{thm}[pro]{Theorem}
\newtheorem{lem}[pro]{Lemma}
\newtheorem{clm}[pro]{Claim}

\newtheorem{cor}[pro]{Corollary}
\newtheorem{quest}[pro]{Question}

\theoremstyle{definition}
\newtheorem{dfn}[pro]{Definition}

\newcommand{\CC}{\mathcal C}

\newcommand{\bdy}{\partial}

\newcommand{\EE}{\mathcal E}
\newcommand{\DD}{\mathcal D}

\newcommand{\TT}{\mathcal T}

\newcommand{\lt}{{\rm left}}
\newcommand{\rt}{{\rm right}}

\newcommand{\plex}[1]{\ensuremath{[{#1}]}}

\title{Normalizing Topologically Minimal Surfaces III: Bounded Combinatorics}
\date{\today}
\author{David Bachman}
\thanks{Partially supported by NSF grant DMS-1207804.}

\begin{document}

\begin{abstract}
We show that there are a finite number of possible pictures for a surface in a tetrahedron with local index $n$. Combined with previous results, this establishes that any topologically minimal surface can be transformed into one with a particular normal form with respect to any triangulation.
\end{abstract}
\maketitle

\markright{NORMALIZING TOPOLOGICALLY MINIMAL SURFACES III}

\section{Introduction}

A surface in a triangulated 3-manifold is {\it normal} if it intersects each tetrahedron transversally in a collection of flat triangles and quadrilaterals. H. Kneser introduced normal surfaces to show that the prime decomposition of a 3-manifold is finite \cite{kneser:29}. Many years later, W. Haken used normal surfaces to give the first algorithm to determine if a knot is the unknot, and to show that any collection of pairwise disjoint, non-parallel incompressible surfaces is finite \cite{haken:61} (see also \cite{finite}). The key to this latter proof was to show that any incompressible surface can be isotoped to a normal surface. 

In 1993, Rubinstein introduced the idea of an {\it almost normal} surface. Such a surface meets each tetrahedron of a triangulated 3-manifold in a collection of normal disks (triangles or quadrilaterals), with the single exception of a disk whose boundary meets the 1-skeleton eight times (i.e. an octagon) or two normal disks connected by a single unknotted tube. Rubinstein \cite{rubinstein:95} and Thompson \cite{thompson:94} used almost normal surfaces to produce an algorithm to determine if a given 3-manifold is $S^3$. Subsequently Rubinstein \cite{rubinstein:93} and Stocking \cite{stocking:96} used almost normal surfaces to produce an algorithm to determine the Heegaard genus of a given 3-manifold. More recently, Li used almost normal surfaces to prove the generalized Waldhausen conjecture \cite{li:07}: any closed, atoroidal 3-manifold admits finitely many Heegaard splittings of each genus. The key to these results is to show that any {\it strongly irreducible} \cite{cg:87} Heegaard splitting is isotopic to an almost normal surface. 

These results illustrate the idea that a ``normal form" for a particular class of surfaces with respect to an arbitrary triangulation enables proofs of both algorithmic and finiteness results. In \cite{TopIndexI}, the author introduced the idea of {\it topologically minimal surfaces} as a way of generalizing both incompressible and strongly irreducible surfaces. Such surfaces have an {\it index}, with index 0 surfaces being incompressible and index 1 surfaces being strongly irreducible. Given the aforementioned results, the obvious question is, ``Can every topologically minimal surface be isotoped to some normal form?" Resolving this question is the purpose of the present series of papers. 

In the first paper in this series \cite{TopMinNormalI} we reduced the problem to a local one. In that work, we showed that a topologically minimal surface with index $n$ in a triangulated 3-manifold can be transformed\footnote{Such a transformation may involve isotopy, $\bdy$-compression, and some compressions. Compressions that arise are very constrained, only happen when $n$ is at least three, and reduce the index.} to one that meets each tetrahedron in a surface with a well-defined {\it local index}. Furthermore, the sum of these local indices, taken over all tetrahedra, will be at most $n$. What remains, then, is to classify all surfaces in a tetrahedron with well-defined local index.

In the second paper in this series \cite{TopMinNormalII} we obtained a complete classification of simply connected surfaces in a tetrahedron with local index $n$. There we showed that such a surface must look like a helicoid, and the boundary of such a surface will meet the 1-skeleton precisely $4(n+1)$ times. This result bears a striking resemblance with Colding and Minicozzi's classification of the local picture of a geometrically minimal surface that is simply connected \cite{cm1}, \cite{cm2}, \cite{cm3}, \cite{cm4}. 

In this final paper, we study the general case of surfaces in a tetrahedron $\Delta$ with local index $n$. Our goal here is to show that there is a finite number of possibilities for such a surface. The results of this paper can all be summarized in the following theorem.

\begin{thm}
\label{t:summary}
Let $H$ be a surface in a tetrahedron $\Delta$ with local index $n$. Then 
	\begin{enumerate}
		\item $H$ is unknotted and the components of $H$ are not nested.
		\item $\sum 1-\chi(H_i) \le n,$ where the sum is taken over all components $H_i$ of $H$.
		\item Each component of $H$ has well-defined local index, and the sum of these indices is precisely $n$. 
		\item If $H'$ is a component of $H$ then $|\bdy H'| \le 4(n'+1),$  where $n'$ is the local index of $H'$. 
	\end{enumerate}
\end{thm}

The results of this theorem imply that there are a finite number of possibilities for a surface in a tetrahedron with local index $n$. For example, when $n=0$ the theorem implies that the surface $H$ must be a collection of disks with boundary length 3 or 4. This recovers the proof that any incompressible surface can be isotoped to a normal surface. When $n=1$ the above result implies that there is at most one component of $H$ that is not a normal triangle or quadrilateral. This component can either be a disk or an annulus, and the total length of its boundary is at most 8. Hence, the exceptional component is either a disk with boundary length 8, or an unknotted annulus connecting two loops that have length 3 or 4. Taken with the results from the prequels, this recovers Rubinstein and Stocking's result that any strongly irreducible surface can be isotoped to an almost normal one. 

The index 2 case is of particular interest, as it has immediate consequences that were not previously known. Theorem \ref{t:summary} implies that a (connected) surface in a tetrahedron with local index $2$ is either:
	\begin{enumerate}
		\item A disk with boundary of length 12 (i.e.~a 12-gon). 
		\item Three normal disks, connected by two unknotted tubes, to form a pair-of-pants.
		\item An octagon connected to a normal triangle by an unknotted tube. 
		\item An octagon connected to itself by an unknotted tube to form a punctured torus. 
		\item A 12-gon connected to itself by an unknotted tube to form a punctured torus. 
	\end{enumerate}
	
Although Theorem \ref{t:summary} does not rule out this last surface as a possible index 2 picture, we conjecture that it is not. This illustrates an important point: Theorem \ref{t:summary} only gives a finite set of surfaces that contains all possible ones that have local index $n$. However, it is likely the case that not all of the surfaces in this set will have the desired index. Despite this, the theorem does establish that for each local index the possibilities are finite, which is what is necessary for applications. For example taken with Theorem 1.2 of \cite{TopMinNormalI} and Lemma 2.10 of \cite{JSJ}, we obtain the following result:

\begin{cor}
\label{c:FiniteSlopes}
Let $M$ be a compact, irreducible 3-manifold with incompressible, connected, toroidal boundary. Then there is a finite set of slopes $\Omega$ on $\bdy M$ such that the boundary slope of any non-peripheral, topologically minimal, index $n$ surface in $M$ is at a distance of at most $n$ (as measured in the Farey graph) from $\Omega$.
\end{cor}

\begin{proof}
Choose any triangulation $\TT$ of $M$. Let $H$ be a topologically minimal surface of index $n$ in $M$. By Theorem 1.2 of \cite{TopMinNormalI}, if $H$ is not peripheral then we can isotope, compress and $\bdy$-compress $H$ to a surface $H'$ such that the local indices of $H' \cap \Delta_i$ sum to at most $n$, where $\{\Delta_i\}$ are the tetrahedra of $\TT$. Furthermore, the slope of $\bdy H'$ will be at most distance $n$ (as measured in the Farey graph) from the slope of $\bdy H$ (see Remark 2.18 of \cite{TopMinNormalI}). By Theorem \ref{t:summary}, for each $i$, each component of the surface $H' \cap \bdy \Delta_i$ is of bounded length, and thus $H'$ is in one of a finite number of {\it compatibility classes} (see Definition 2.5 of \cite{JSJ}). Lemma 2.10 in \cite{JSJ} implies that each such compatibility class determines a unique slope on $\bdy M$ (see in particular the paragraph following the proof of this lemma).
\end{proof}

This result, together with Theorem 4.9 of \cite{TopIndexI} then implies the following, which says roughly that if manifolds $X$ and $Y$ with toroidal boundary are glued together by a ``sufficiently complicated" map, then no new index $n$ surfaces are created. 

\begin{cor}
\label{c:Barrier}
Let $X$ and $Y$ be compact, irreducible 3-manifolds with incompressible, connected, toroidal boundary. Let $\phi:\bdy X \to \bdy Y$ be an homeomorphism, and $\psi:\bdy Y \to \bdy Y$ be any Anosov map. Then there is a value $m_0$ depending only on $n$ such that for all $m \ge m_0$, any topologically minimal, index $n$ surface in the manifold $X \cup _{\psi^m \circ \phi} Y$ is isotopic into $X$ or $Y$. 
\end{cor}

\begin{proof}
Let $\Omega_X$ and $\Omega_Y$ be the finite sets of slopes on $\bdy X$ and $\bdy Y$ guaranteed by Corollary \ref{c:FiniteSlopes}, so that if $H_X \subset X$ and $H_Y \subset Y$ are surfaces with topological index $n$, then their boundaries are at most distance $n$ from $\Omega_X$ and $\Omega_Y$, respectively. Choose $m_0$ so that if  $m \ge m_0$, then the set $\Omega_Y$ is at least distance $2n+1$ from the set $\psi^m \circ \phi(\Omega_X)$ in the Farey graph of $\bdy Y$. 

Now suppose $H$ is an index $n$ surface in $M=X \cup _{\psi^m \circ \phi} Y$ that is not isotopic into either $X$ or $Y$. Let $T$ be the image of $\bdy Y$ in this manifold. By Theorem 4.9 of \cite{TopIndexI}, $H$ may be isotoped to meet $T$ transversally away from $p$ saddle points, so that the sum of the indices of the components of $H-N(T)$ is at most $n-p$. In particular, there are components $H_X \subset X$ and $H_Y \subset Y$ of $H-N(T)$ whose indices are at most $n$. Furthermore, $\bdy H_X$ and $\bdy H_Y$ are essential on $T$, and come from $H \cap T$ by resolving $p$ saddle points. It follows that these loops are disjoint, and since $T$ is a torus, must actually be the same. 

We now apply Corollary \ref{c:FiniteSlopes}, which says that $\bdy H_X$ is at a distance of at most $n$ from $\Omega_X$, and $\bdy H_Y$ is distance at most $n$ from $\Omega_Y$. Since $\psi^m \circ \phi(\bdy H_X)=\bdy H_Y$, we conclude the distance between $\Omega_Y$ and $\psi^m \circ \phi(\Omega_X)$ is at most $2n$, contradicting our choice of $m_0$. 
\end{proof}

In Corollary 11.2 of \cite{Stabilizing} we constructed the first example of a manifold that contains a non-minimal genus, unstabilized, weakly reducible Heegaard splitting. By using Corollary \ref{c:Barrier}, we can improve this construction to show that there is a manifold that contains an infinite number of such splittings. 

\begin{cor}
There is a closed 3-manifold $M$ that contains infinitely many non-isotopic, unstabilized Heegaard splittings that are weakly reducible (i.e.~splittings whose Hempel distance is precisely one). 
\end{cor}

\begin{proof}
In the parlance of \cite{Stabilizing}, Corollary \ref{c:Barrier} establishes that if $X$ and $Y$ are manifolds with toroidal boundary that are glued by a sufficiently complicated homeomorphism, then the torus $T$ at their interface is a ``$g$-barrier surface" for all index 0, 1, and 2 surfaces, regardless of the genus $g$. Thus, the conclusion of Theorem 11.1 of \cite{Stabilizing} holds, regardless of the genus of the surface $H_i$. The construction given in Corollary 11.2 of \cite{Stabilizing} then works for every genus. 
\end{proof}

\noindent {\bf Acknowledgements.} The author is immensely grateful to Ryan Derby-Talbot and Eric Sedgwick for invaluable conversations during the preparation of this work.

\section{Topologically Minimal Surfaces in Tetrahedra}
\label{s:TopMinDefinitions}

Throughout this paper $H$ will represent a properly embedded surface in a tetrahedron $\Delta$. The 1- and 2-skeleton of $\Delta$ will be denoted $\TT^1$ and $\TT^2$, respectively. 

\begin{dfn}
A compressing disk for $H$ is a disk $C$ such that $C \cap H$ is an essential loop on $H$. We will use the notation $\CC(H)$ to refer to the set of compressing disks for $H$. 
\end{dfn}

\begin{dfn}
An edge-compressing disk for $H$ is a disk $E$ such that $\bdy E=\alpha \cup \beta$, where $E \cap H=\alpha$ and $E \cap \TT^1 =\beta$ is a subarc of an edge of $\TT^1$. We will use the notation $\EE(H)$ to refer to the set of edge-compressing disks for $H$. 
\end{dfn}

\begin{dfn}
We let $\CC \EE(H)=\CC(H) \cup \EE(H)$. 
\end{dfn}

\begin{dfn}
\label{d:H/D}
For any disk $D \in \CC \EE(H)$, we define $H/D$ to be the surface obtained from $H$ as follows.  Let $N(D)$ denote a submanifold of $\Delta$ homeomorphic to $D \times I$, so that $D=D \times  \{\frac{1}{2}\}$, $N(D) \cap H = (D \cap H) \times I$, and  $N(D) \cap \bdy \Delta = (D \cap \bdy \Delta) \times I$. Then $H/D$ is obtained from $H$ by removing $N(D) \cap H$ and replacing it with $D \times \bdy I$. The two subdisks  $D \times \bdy I$ in $H/D$ will be called the {\it scars} of $D$. When $\mathcal D$ is a collection of pairwise disjoint disks in $\CC \EE(H)$, then $H/\DD$ will denote the result of simultaneously performing the above operation on each individual disk in $\mathcal D$. 
\end{dfn}

We define an equivalence relation as follows:

\begin{dfn}
Disks $D,D' \in \CC \EE(H)$ are {\it equivalent}, $D \sim D'$, if $D$ and $D'$ are isotopic through disks in $\CC \EE(H)$. 
\end{dfn}

\begin{dfn}
\label{d:DiskComplex}
The {\it disk complex} $\plex{\CC \EE(H)}$ is the complex defined as follows: vertices correspond to equivalence classes $[D]$ of the disk set $\CC \EE(H)$. A collection of $n$ vertices spans an $(n-1)$-simplex if there are representatives of each which are pairwise disjoint away from a neighborhood of $\TT^1$.
\end{dfn}

\begin{dfn}
\label{d:Indexn}
The surface $H$ is said to be {\it topologically minimal in the tetrahedron $\Delta$} if the complex $\plex{\CC \EE(H)}$ is either empty or non-contractible. In the former case we define the {\it local index} of $H$ to be 0. In the latter case, the {\it local index} of $H$ is defined to be the smallest $n$ such that  $\pi_{n-1}(\plex{\CC \EE(H)})$ is non-trivial. 
\end{dfn}

The following definition and lemma will aid us in establishing the third conclusion of Theorem \ref{t:summary}.

\begin{dfn}
We say two components of a surface $H \subset \Delta$ are {\it normally separated} if there is a normal triangle or quadrilateral in $\Delta$ separating them. The {\it normal components} of $H$ are then the disjoint unions of the components of $H$ that are not normally separated. A surface is {\it normally connected} if it has one normal component. (Equivalently, no two components of $H$ can be normally separated.) 
\end{dfn}

For example, if the components of $H$ are two parallel octagons and two normal triangles, then $H$ has three normal components. One of these normal components will contain the two octagons, and the other two will each contain a single normal triangle. 

\begin{lem}
If a surface $H$ in a tetrahedron $\Delta$ has local index $n$, then the normal components of $H$ have well-defined local indices, and the sum of these indices is precisely $n$. 
\end{lem}

\begin{proof}
Suppose $H=H_1 \cup H_2$, where $H_1$ and $H_2$ are subsurfaces that are separated by a normal triangle or quadrilateral. Then every compressing and edge-compressing disk for $H$ is a compressing or edge compressing disk for either $H_1$ or $H_2$, and vice versa. Hence, the complex $[\CC \EE(H)]$ is the join of the complexes $[\CC \EE(H_1)]$ and $[\CC \EE(H_2)]$. It follows that both $H_1$ and $H_2$ have well-defined local indices, and the sum of these indices is precisely the local index of $H$ (see, for example, the proof of Theorem 4.7 of \cite{TopIndexI}). The lemma now follows by inducting on the number of normal components of $H$. 
\end{proof}

Following this, to establish the  third conclusion of Theorem \ref{t:summary} we only need to show that a normally connected surface $H$ with well-defined local index is actually connected. Hence, in the remaining sections we will assume the surface $H$ under consideration is normally connected. In particular, $H$ has no components that are normal triangles or quadrilaterals.

\section{The topology of topologically minimal surfaces in tetrahedra}
\label{s:TopologyConnected}

In this section we limit the topology of a topologically minimal surface $H$ in a tetrahedron $\Delta$ and establish the first two conclusions of Theorem \ref{t:summary}.


\begin{lem}
\label{l:DisjointCompressing}
Let $H$ be a surface in a tetrahedron $\Delta$. If $\EE$ is a collection of pairwise disjoint edge-compressing disks for $H$, and $C$ is a compressing disk for $H/\EE$, then $C$ is a compressing disk for $H$ that is disjoint from every element of $\EE$. 
\end{lem}

\begin{proof}
To reconstruct $H$ from $H/\EE$, we attach bands between subarcs of $\bdy H/\EE$. As $C$ is in the interior of $\Delta$, and these bands can be chosen to be arbitrarily close to $\bdy \Delta$, the result follows. 
\end{proof}

The first two conclusions of Theorem  \ref{t:summary} follow immediately from the following theorem.

\begin{thm}
\label{t:EdgeCompressions}
Suppose $H$ is an index $n$, topologically minimal surface in a tetrahedron $\Delta$. Then there is a collection $\EE$ of at most $n$ edge-compressing disks for $H$ that are pairwise disjoint away from a neighborhood of $\TT^1$, such that $H/\EE$ is a collection of disks. 
\end{thm}

\begin{proof}
We will say a loop of $\bdy H$ is {\it non-trivial} if it does not bound a disk component of $H$. If $H$ is a collection of disks, the result is trivial. If not, then a component of $\bdy H$ that is innermost on $\bdy \Delta$ among all non-trivial loops bounds a compressing disk for $H$. Thus, $\CC \EE(H) \ne \emptyset$. By definition then, $n \ge 1$ and there is a homotopically non-trivial map $\iota$ of an $(n-1)$-sphere $S$ into $\plex{\CC \EE(H)}$. Triangulate $S$ so that the map $\iota$ is simplicial.

\begin{clm}
There is a vertex $v$ of $S$ such that $\iota(v)$ represents an edge-compressing disk. 
\end{clm}

\begin{proof}
If not, then every vertex in the image of $\iota$ represents a compressing disk. As above, a component of $\bdy H$ that is innermost on $\bdy \Delta$ among all non-trivial loops bounds a compressing disk $C$ for $H$. Furthermore, the disk $C$ can be isotoped to be disjoint from every other compressing disk for $H$. Thus, the vertex $[C]$ of $\plex{\CC \EE(H)}$ is connected by an edge to every vertex in the image of $\iota$. This contradicts the assumption that $\iota$ is homotopically non-trivial. 
\end{proof}

We will call a simplex of $S$ an {\it edge-simplex} if it is spanned by vertices $v$ such that $\iota(v)$ is an edge-compressing disk. Henceforth, we will assume that all choices have been made so that the total number $\epsilon$ of edge-simplices of $S$ (of any dimension) is as small as possible. It follows from the previous claim that $\epsilon>0$.

\begin{clm}
Suppose $\sigma$ is a maximal dimensional edge-simplex of $S$ (i.e. $\sigma$ is not a face of any other edge-simplex). Let $\EE$ be a collection of disks representing the vertices of $\iota(\sigma)$. Then $H/\EE$ is incompressible. 
\end{clm}

\begin{proof}
By way of contradiction, assume $H/\EE$ is compressible. Let $\alpha$ be a component of $\bdy (H /\EE)$ that is innermost on $\bdy \Delta$, among all loops that do not bound disk components of $H/\EE$. Then $\alpha$ bounds a compressing disk $C$ for $H/\EE$ that can be made disjoint from every other compressing disk for $H/\EE$. By Lemma \ref{l:DisjointCompressing}, $C$ is also a compressing disk for $H$ that is disjoint from every disk representing the vertices of $\iota(\sigma)$.

Let $\Sigma$ denote the union of the simplices in $S$ that contain $\sigma$. Thus, each vertex $v$ of a simplex of $\Sigma$ is either contained in $\sigma$, or is connected by an edge to every vertex of $\sigma$. In the latter case, $\iota(v)$ must be represented by a compressing disk $V$ for $H$, as otherwise $\sigma \cup v$ would be an edge-simplex containing $\sigma$, contradicting our assumption that $\sigma$ was of maximal dimension. As $v$ is connected by an edge to each vertex of $\sigma$ it must be disjoint from every edge-compressing disk in $\EE$. Thus, the disk $V$ is also a compressing disk for the surface $H/\EE$. Since $C$ is disjoint from every compressing disk for $H/\EE$, it must be the case that either $[C]=[V]$ in $\plex{\CC \EE(H)}$, or $[C]$ is connected to $[V]$ by an edge. In either case, this shows that we can alter the map $\iota$ by replacing $\Sigma$ with the cone of $\bdy \Sigma$ to a vertex $c$, where $\iota(c)=[C]$. This replacement removes the edge-simplex $\sigma$ without introducing any new edge-simplices, thus lowering $\epsilon$. 
\end{proof}

The theorem now follows from the previous claims. 
\end{proof}

The following lemmas are an important step in the proof of the third conclusion of Theorem \ref{t:summary}.

\begin{lem}
\label{l:SameComponent}
Suppose $H$ is a normally connected surface in a tetrahedron $\Delta$ with well-defined local index and $A$ is a component of $\bdy \Delta -\bdy H$. If there is a loop $\alpha$ in $A$ of length three or four, then there are boundary components of $A$ on either side of $\alpha$ that belong to the same component of $H$. 
\end{lem}

\begin{proof}
Let $C$ be a normal triangle or quadrilateral bounded by $\alpha$. Isotope $H$ so that $|C \cap H|$ is minimal. Let $\EE$ denote a collection of edge-compressing disks for $H$ such that $H/\EE$ is a collection of disks, as guaranteed to exists by Theorem \ref{t:EdgeCompressions}. We claim that some element $E \in \EE$ meets $\alpha=\bdy C$. If not, then there will be a non-trivial loop in $C - (H \cup \EE)$. Let $\beta$ denote an innermost loop of a maximal collection of pairwise disjoint, non-parallel, non-trivial loops in $C - (H \cup \EE)$. As $H/\EE$ is a collection of disks, the loop $\beta$ must bound a disk $D$ in $\Delta - H/\EE$. The surface $H$ can be recovered from $H /\EE$ by attaching bands that can be chosen to be arbitrarily close to $\bdy \Delta$. In particular, these bands can be chosen to be disjoint from $D$. We conclude the disk $D$ is disjoint from $H$. It follows that isotoping $D$ into $C$ reduces $|H \cap C|$, contradicting minimality. We conclude some element $E \in \EE$ meets $\alpha$. The arc $E \cap H$ will then be a path in $H$ connecting points of $\bdy A$ on opposite sides of $\alpha$.
\end{proof}

\begin{lem}
\label{l:NoLongConnected}
If $H$ is a normally connected surface in a tetrahedron $\Delta$ with well-defined local index and every component of $\bdy H$ has length three or four, then $H$ is connected. 
\end{lem}

\begin{proof}
Each component $A$ of $\bdy \Delta -\bdy H$ is either an annulus, ``pair-of-pants" (i.e. thrice-punctured sphere), or ``t-shirt" (i.e. four times punctured sphere). It suffices to show that in all cases, every component of $\bdy A$ is on the same component of $H$. 

Consider first the case when $A$ is an annulus. Then a loop $\alpha$ on $A$ that is normally parallel to either boundary component will have length three or four. Applying Lemma \ref{l:SameComponent} to the loop $\alpha$ then tells us that both components of $\bdy A$ belong to the same component of $H$. 

When $A$ is a pair-of-pants, let $A_1$, $A_2$, and $A_3$ denote the boundary components of $A$, and $\alpha_i$ a loop in $A$ that is normally parallel to $A_i$. Applying Lemma \ref{l:SameComponent} to $\alpha_1$, we conclude that $A_1$ is on the same component of $H$ as either  $A_2$ or $A_3$. Suppose the former. We then apply Lemma \ref{l:SameComponent} to the loop $\alpha_3$ to conclude $A_3$ is on the same component of $H$ as $A_1$ and $A_2$. Thus, all three boundary components of $A$ are on the same component of $H$. 

We now deal with this case when $A$ is a ``t-shirt". This arrises when all three components of $\bdy A$ are length three. Let $A_1$, $A_2$, $A_3$, and $A_4$ denote the boundary components of $A$, and $\alpha_i$ a loop in $A$ that is normally parallel to $A_i$. Applying Lemma \ref{l:SameComponent} to $\alpha_1$, we conclude that $A_1$ is on the same component of $H$ as either  $A_2$, $A_3$, or $A_4$. Without loss of generality, assume $A_2$. We now apply Lemma \ref{l:SameComponent} to the loop $\alpha_3$ to conclude $A_3$ is on the same component of $H$ as $A_1$, $A_2$, or $A_4$. There are two cases to consider:

\bigskip

\noindent {\it Case 1.} $A_3$ is on the same component of $H$ as $A_1$ and $A_2$. Then apply Lemma \ref{l:SameComponent} to the loop $\alpha_4$ to conclude $A_4$ is on the same component of $H$ as $A_1$, $A_2$, and $A_3$. 

\bigskip

\noindent {\it Case 2.} $A_3$ is on the same component of $H$ as $A_4$. Let $\alpha$ be a loop on $A$ which separates $A_1$ and $A_2$ from $A_3$ and $A_4$. Then $\alpha$ has length four, so we may apply Lemma \ref{l:SameComponent} to it. We conclude $A_1$ and $A_2$ are on the same component of $H$ as $A_3$ and $A_4$. 
\end{proof}

\section{Bounding the length of boundary curves with no triangular components}

In this section we establish the fourth conclusion of Theorem \ref{t:summary}, assuming $\bdy H$ contains only loops of length at least 8. In this case, it is well known that $\TT^1$, the 1-skeleton of $\Delta$, consists of three pairs of opposite edges, $s$ and $s'$, $m$ and $m'$, and $l$ and $l'$, such that
		\begin{enumerate}
			\item $|\bdy H \cap s|=|\bdy H \cap s'|=a>0$
			\item $|\bdy H \cap m|=|\bdy H \cap m'|=b \ge a$
			\item $|\bdy H \cap l|=|\bdy H \cap l'|=c=a+b$
		\end{enumerate}
The length of $\bdy H$ is precisely $2(a+b+c)=4c$. Hence, an upper bound on $c$ provides us with an upper bound on $|\bdy H|$. We will call $s$ and $s'$ the {\it short edges}, $m$ and $m'$ the {\it medium edges}, and $l$ and $l'$ the {\it long edges}. Note that when $a=b$ the choice of short and medium edges can be made arbitrarily. 

For the remainder of this section $v$ will denote the vertex of $\Delta$ where the edges $s$, $m$, and $l$ meet. For any edge $e$ incident to the vertex $v$, we can identify the frontier $R_e$ of a neighborhood $N(e)$ of $e$ in $\Delta$ with $I \times I$ so that 
	\begin{itemize}
		\item for each $t \in I$, the arc $t \times I$ is parallel in $N(e)$ to $e$, and 
		\item for each component $\gamma$ of $H \cap R$ there is an $s \in I$ such that $\gamma=I \times s$. 
	\end{itemize}

\begin{dfn}
Let $\tau^e_\lt$ and $\tau^e_\rt$ denote the 2-simplices of $\Delta$ that are incident to the edges $0 \times I$ and $1 \times I$ of $R_e$, respectively. We say these are the 2-simplices that are to the {\it left} and {\it right} of $e$. For the three edges $s$, $m$, and $l$, we assume identifications of $R_s$, $R_m$, and $R_l$ with $I \times I$ have been chosen so that $\tau^m_\rt=\tau^l_\lt$, $\tau^l_\rt=\tau^s_\lt$, and $\tau^s_\rt=\tau^m_\lt$. 
\end{dfn}

\begin{dfn}
A subregion of $R_e$ cobounded by arcs of $H \cap R_e$, which does {\it not} meet a region of $\tau^e_\lt$ bounded by parallel normal arcs of $\bdy H$ is called a {\it left switch of $e$}. See Figure \ref{f:Switches}. Right switches are defined similarly. 
\end{dfn}

\begin{figure}
\psfrag{R}{$R_e$}
\psfrag{T}{$\tau^e_\lt$}
\psfrag{t}{$\tau^e_\rt$}
\psfrag{l}{Left switch}
\psfrag{r}{Right switch}
\[\includegraphics[width=3.5in]{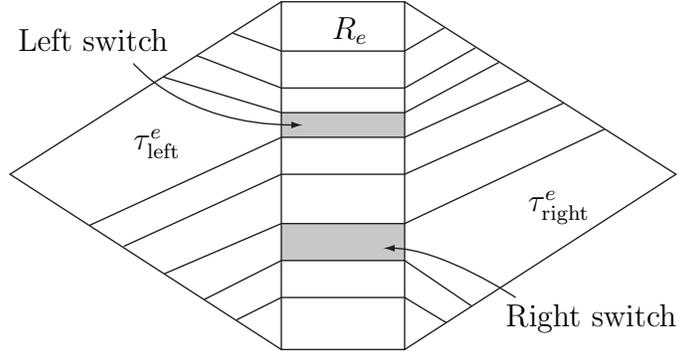}\]
\caption{Defining left and right {\it switches}.}
\label{f:Switches}
\end{figure}

\begin{dfn}
Let $E$ denote an edge-compressing disk for $H$ incident to an edge $e$. Suppose the arcs of $H \cap (R_e \cup \tau^e_\lt)$ that meet $E$ are parallel, as in Figure \ref{f:PushOff}. Let $E'$ denote the union of the shaded region in the figure, together with the disk $E$. Then $E'$ is an edge-compressing disk for $H$, called the {\it push-off of $E$ to its left}. Push-offs to the right are defined similarly. 
\end{dfn}

\begin{figure}
\psfrag{R}{$R_e$}
\psfrag{F}{$E'$}
\psfrag{E}{$E \cap R_e$}
\psfrag{S}{$\tau^e_\lt$}
\[\includegraphics[width=3in]{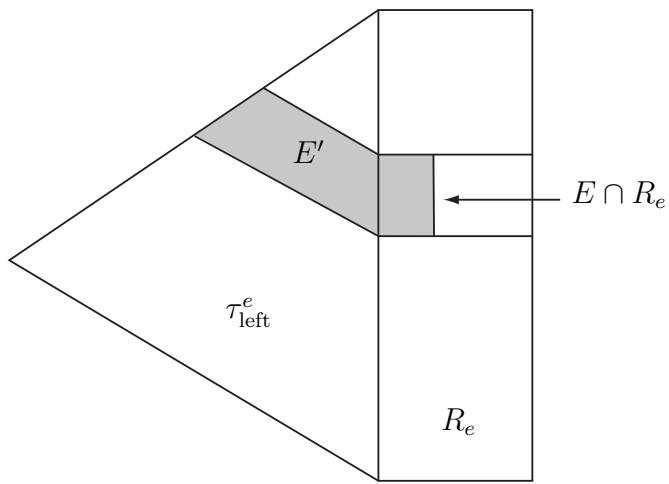}\]
\caption{Defining the push-off $E'$ of an edge-compressing disk $E$.}
\label{f:PushOff}
\end{figure}

The push-off operation can sometimes be iterated, as follows:

\begin{dfn}
Let $E$ denote an edge-compressing disk for $H$ incident to an edge $e$. Let $f$ denote an edge of $\Delta$, so that $\tau^f_\rt=\tau^e_\lt$. Suppose the arcs of $H \cap (R_e \cup \tau^e_\lt \cup R_f \cup \tau^f_\lt)$ that meet $E$ are parallel, as in Figure \ref{f:SecondPushOff}. Let $E''$ denote the union of the shaded region in the figure, together with the disk $E$. Then $E''$ is an edge-compressing disk for $H$, called the {\it second push-off of $E$ to its left}. Second push-offs to the right are defined similarly. 
\end{dfn}

\begin{figure}
\psfrag{R}{$R_e$}
\psfrag{r}{$R_f$}
\psfrag{F}{$E''$}
\psfrag{E}{$E \cap R_e$}
\psfrag{S}{$\tau^e_\lt=\tau^f_\rt$}
\psfrag{s}{$\tau^f_\lt$}
\[\includegraphics[width=3in]{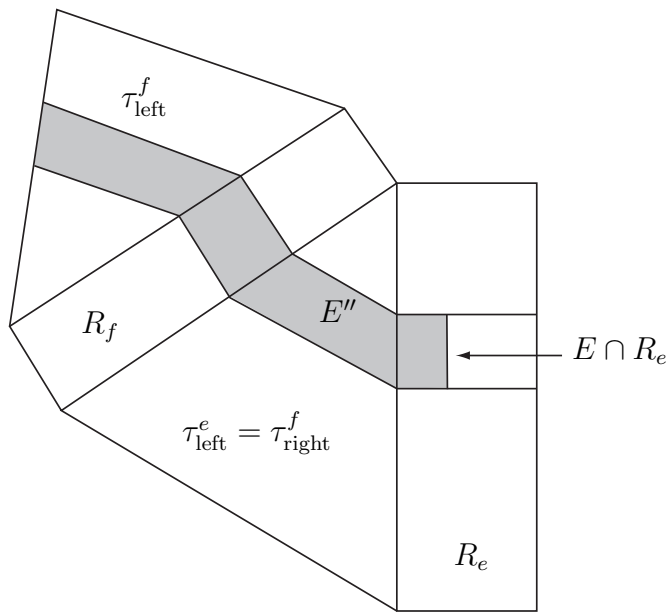}\]
\caption{The second push-off $E''$ of $E$.}
\label{f:SecondPushOff}
\end{figure}

\begin{dfn}
Let $\EE$ denote a pairwise disjoint collection of edge-compressing disks incident to $e$.  Isotope the disks in $\EE$ so that for each $E \in \EE$, there is a distinct $t_E \in I$ such that $E \cap R_e \subset t_E \times I$. For each disk $E \in \EE$, let $\gamma_E^{\pm}$ denote the arcs of $H \cap R_e$ incident to $E \cap R_e$. We say $E$ is a {\it leftmost} disk of $\EE$ if for each $E' \in \EE$ that meets $\gamma_E^+$ or $\gamma_E^-$, $t_E < t_{E'}$. See Figure \ref{f:Leftmost}. Rightmost disks are defined similarly. 
\end{dfn}

\begin{figure}
\psfrag{R}{$R_e$}
\psfrag{H}{$H \cap R_e$}
\psfrag{l}{Leftmost}
\psfrag{r}{Rightmost}
\[\includegraphics[width=2in]{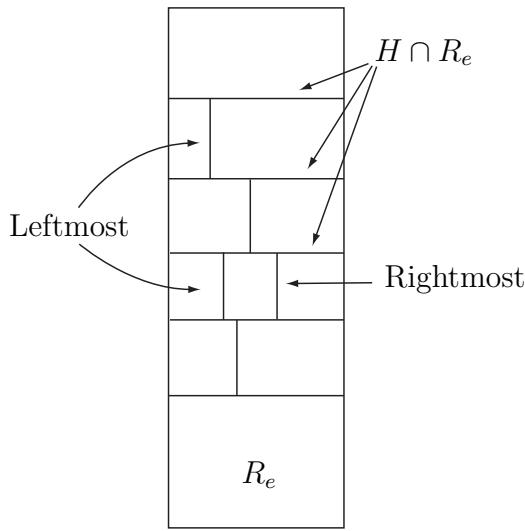}\]
\caption{Defining {\it leftmost} and {\it rightmost}.}
\label{f:Leftmost}
\end{figure}

Note that it follows immediately from the above definitions that if $E$ is a leftmost disk of a collection $\EE$, and $E'$ is the push-off of $E$ to its left, then $E'$ is disjoint from every disk in $\EE$.

As $H$ has local index $n$, there is a homotopically non-trivial map $\iota$ from an $(n-1)$-sphere $S$ into $[\CC \EE(H)]$. Triangulate $S$ so that the map $\iota$ is simplicial. 

\begin{dfn}
Let $e$ be an edge of $\Delta$. An $e$-simplex of $S$ is a simplex spanned by vertices that represent edge-compressing disks incident to $e$. Such a simplex is {\it full} if it is not a face of another $e$-simplex. 
\end{dfn}

\begin{lem}
\label{l:NoShort}
$S$ is homotopic to a sphere that contains no vertices representing edge-compressing disks incident to any short or medium edge.
\end{lem}

\begin{proof}
Choose $\iota$ and the triangulation of $S$ so that the sum $\epsilon$ of the total number of $s$-, $s'-$, $m-$, and $m'-$simplices is minimal. Our goal is to show that $\epsilon=0$. If not, then $S$ contains an $e$-simplex for some short or medium edge $e$. Thus, $S$ contains a full $e$-simplex, $\sigma$. Let $\EE$ denote a collection of disks represented by the vertices of $\sigma$.

Note that it follows from the fact that $e$ is a short or medium edge, and that there are no triangular components of $\bdy H$, that there are no switches on $e$. Hence, if $\EE$ contains a disk incident to $e$ then any rightmost disk $E$ of $\EE$ on $e$ has a push-off to the right. Such a push-off $E'$ will be incident to a long edge and will be disjoint from every disk in $\EE$. 

Let $\Sigma$ denote the union of the simplices in $S$ that contain $\sigma$. Suppose $[X]$ is a vertex of $\Sigma$ that is not in $\sigma$. If $X$ meets the edge $e$, then the join of $\sigma$ with $[X]$ would be an $e$-simplex containing $\sigma$ as a face. Since $\sigma$ is full, no such simplex exists. We conclude $X$ does not meet $e$. Furthermore, as $[X] \in \Sigma$, it is connected to every vertex of $\sigma$ by an edge. Thus, the disk $X$ is disjoint from every element of $\EE$. It follows that $X$ will be disjoint from the push-off $E'$, as this disk is constructed from $E$ (which is disjoint from $X$) and a subset of $\bdy \Delta$. 

As $E'$ is disjoint from every element of $\EE$ and every disk represented by a vertex of $\Sigma-\sigma$, we conclude the vertex $[E']$ is connected by an edge to every vertex of $\Sigma$ in $[\CC \EE(H)]$. Thus, we may replace $\Sigma$ in $S$ with the cone of $\bdy \Sigma$ to the point $[E']$. Since $\sigma$ is no longer an $e$-simplex of the resulting sphere, and no new simplices have been created spanned by edge-compressing disks incident to any other short or medium edge, we have lowered $\epsilon$.
\end{proof}

By Lemma \ref{l:NoShort}, we may henceforth assume that $S$ has no vertices representing edge-compressing disks incident to any short or medium edge. We now assume that in addition, $S$ has been chosen so as to minimize the total number $\delta$ of $l$-simplices.

\begin{lem}
\label{l:LongExists}
There is an $l$-simplex of $S$, and every full such simplex contains vertices representing edge-compressing disks that meet the points of $H \cap l$ that are closest to the endpoints of $l$.
\end{lem}

\begin{proof}
Let $v^*$ denote one of the endpoints of $l$. Let $\Delta_{v^*}$ denote the union of the three 2-simplices of $\bdy \Delta$ that meet $v^*$. Let $s^*$, $m^*$, and $l^*$ denote the points of $\bdy H$ that are closest to $v^*$, on the short, medium, and long edges incident to $v^*$. Let $E$ be the closure of the component of $\Delta_{v^*} - H$ that contains the vertex $v^*$. See Figure \ref{f:Delta_v}. Then pushing $E$ slightly into $\Delta$ turns it into an edge-compressing disk for $H$ incident to the edge $l'$ opposite $l$. As $E$ is isotopic into $\bdy \Delta$, it is disjoint from every compressing disk for $H$, and every edge-compressing disk for $H$ that does not meet $s^*$, $m^*$, or $l^*$. Thus, if $S$ contains no vertices representing edge-compressing disks that meet $s^*$, $m^*$, and $l^*$, then $\iota(S)$ is homotopic to the point $[E]$, a contradiction. However, by assumption, there are no edge-compressing disks represented by vertices of $S$ that are incident to either $s$ or $m$. We conclude there is a vertex of $S$ that represents an edge-compressing disk incident to $l^*$.

\begin{figure}
\psfrag{L}{$l$}
\psfrag{l}{$l'$}
\psfrag{s}{$s^*$}
\psfrag{m}{$m^*$}
\psfrag{K}{$l^*$}
\psfrag{v}{$v^*$}
\psfrag{E}{$E$}
\[\includegraphics[width=3in]{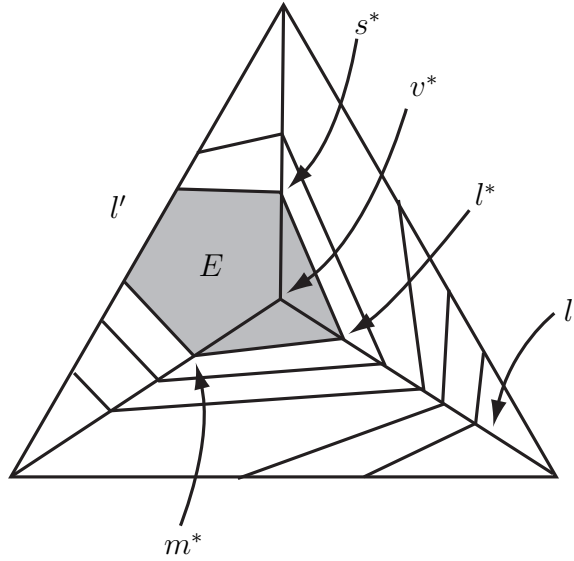}\]
\caption{The three 2-simplices of $\bdy \Delta$ that meet $v^*$, and the edge-compressing disk $E$ incident to $l'$.}
\label{f:Delta_v}
\end{figure}

The above argument implies there is a vertex of $S$ representing a disk that meets the edge $l$. Hence, there is an $l$-simplex, and thus a full $l$-simplex, $\sigma$. Let $\Sigma$ denote the union of the simplices of $S$ that contain $\sigma$. Then no vertex of $\Sigma-\sigma$ represents a disk incident to $l$, and hence all such vertices are connected by an edge to $[E]$ in $[\CC \EE(H)]$. Similarly, if no vertex of $\sigma$ represents a disk that meets the point $l^*$, then all such vertices are also connected by an edge to $[E]$ in $[\CC \EE(H)]$. We conclude that in this case we will be able to lower $\delta$ by replacing $\Sigma$ with the cone on $\bdy \Sigma$ to the point $[E]$. This contradiction establishes that any full $l$-simplex of $S$ contains an edge-compressing disk incident to $l^*$. 
\end{proof}

\begin{thm}
\label{t:NoLong}
If $H$ is an index $n$ topologically minimal surface in a tetrahedron where every component of $\bdy H$ has length at least $8$, then $|\bdy H| \le 4(n+1)$. 
\end{thm}

Note that when $H$ is simply connected, then in \cite{TopMinNormalII} we have already established that $|\bdy H|$ is precisely $4(n+1)$.

\begin{proof}
By Lemma \ref{l:LongExists}, there is a full $l$-simplex $\sigma$ of $S$. Let $\EE$ denote a collection of disks representing the vertices of $\sigma$. We claim that each point of $l$ between two points of $H \cap l$ meets some disk in $\EE$. Suppose not. Let $p$ be a point of $l$ between points of $H \cap l$ that does not meet any disk in $\EE$.  By Lemma \ref{l:LongExists}, there are disks in $\EE$ that meet $l$ on both sides of $p$. Thus, there is both a leftmost and a rightmost disk of $\EE$ on both sides of $p$. As there is only one left switch on $l$, there must then be a leftmost disk of $\EE$ that does not meet the left switch. Such a disk $E$ will have a second push-off to a disk $E'$ that meets the edge $l'$, where $E'$ is disjoint from every disk in $\EE$. 

The remainder of the proof is similar to that of Lemma \ref{l:LongExists}. Let $\Sigma$ denote the union of the simplices of $S$ that contain $\sigma$. Then by construction, every vertex of $\Sigma$ is connected by an edge to $[E]$ in $[\CC \EE(H)]$. We may therefore lower $c(S;l)$ by replacing $\Sigma$ with the cone on $\bdy \Sigma$ to the point $[E]$. 

Since $n$ is the index of $H$, the dimension of $S$ is $n-1$. Thus, every simplex of $S$ is spanned by at most $n$ vertices. Hence, $|\EE| \le n$. The points of $H \cap l$ separate $l$ into $|H \cap l|+1$ intervals. By the above argument, all but the first and last such intervals meets a disk in $\EE$. Thus, $|H \cap l|+1-2 \le n$, or $|H \cap l| \le n+1$. As $l$ is a long edge, $|\bdy H| =4|H \cap l|$. Putting this together then gives us $|\bdy H| \le 4(n+1)$, as desired.  
\end{proof}


The following lemma takes us one more step closer to establishing the third conclusion of Theorem \ref{t:summary}. 

\begin{lem}
\label{l:OnlyLongConnected}
If $H$ is a surface in a tetrahedron with well-defined local index, and every component of $\bdy H$ has length at least $8$, then $H$ is connected. 
\end{lem}

\begin{proof}
Let $S$ be as above. According to the proof of Theorem \ref{t:NoLong},  each point $p$ of $l$ between two points of $H \cap l$ meets some edge-compressing disk $E$ represented by a vertex of $S$. Thus, $E \cap H$ is a path in $H$ connecting the points of $H \cap l$ on either side of $p$. It follows that these points are on the same component of $H$. As this is true for each consecutive pair of points of $H \cap l$, we conclude that all of these points are on the same component of $H$. The result thus follows. 
\end{proof}

\section{Compressing away triangular boundary components}

In this section we assume $H$ is normally connected, and thus has no components that are normal triangles or quadrilaterals. Let $v$ be a vertex of $\Delta$, and suppose $\bdy H$ has a component of length 3 which separates $v$ from the other boundary components of $H$. This length 3 component bounds a subdisk $C_0$ of $\bdy \Delta$ than contains $v$. As $H$ has no components that are normal triangles, $C_0$ is isotopic to a compressing disk for $H$.

\begin{lem}
\label{l:SunburstExists}
Every essential sphere $S$ in $[\CC \EE(H)]$ has a vertex that represents an edge-compressing disk that meets $C_0$. 
\end{lem}

\begin{proof}
As $C_0$ is isotopic into $\bdy \Delta$ it is disjoint from every other compressing disk for $H$. If it is also disjoint from every edge-compressing disk represented by a vertex of $S$, then $S$ is homotopic to the point $[C_0]$ in $[\CC \EE(H)]$.
\end{proof}

Let $\Delta_v$ denote the union of the 2-simplices of $\Delta$ that contain $v$. Let $\Delta^1$ denote $\Delta$ with a neighborhood of the 1-skeleton removed, and $H^1=H \cap \Delta^1$. Let $S$ be an essential sphere in $[\CC \EE(H)]$. Let $[E_0]$ be a vertex of $S$, where $E_0$ is an edge-compressing disk that meets $C_0$ (as guaranteed to exist by Lemma \ref{l:SunburstExists}). Let $\CC_0$ denote the set which contains the single element $C_0$. Let $\EE_0$ be a set of edge-compressing disks that meets $C_0$, includes $E_0$, and represents the vertices of a simplex of $S$.

Now suppose $E_1, E_2 \in \EE_0$. Then $E_1$ and $E_2$ each meet $\bdy C_0$ in a point. If the subarc of $\bdy C_0$ connecting these points does not meet any other disk in $\EE_0$, then we say $E_1$ and $E_2$ are {\it consecutive}. If the subarc of $\bdy C_0$ connecting consecutive disks only meets the edges of $\Delta$ that either $E_1$ or $E_2$ meet, then we say $E_1$ and $E_2$ are {\it neighbors.} Neighboring disks $E_1$ and $E_2$ each meet some other points of $\bdy H^1$. If these two points are connected by a subarc of $\bdy H^1$ that lies in $\Delta_v$, then we say $E_1$ and $E_2$ are {\it adjacent}. Otherwise, $E_1$ and $E_2$ are {\it non-adjacent}.

When $E_1$ and $E_2$ are adjacent neighbors, then the arcs $E_1 \cap \bdy \Delta^1$ and $E_2 \cap \bdy \Delta^1$, together with two subarcs of $\bdy H^1$, bound a rectangle $Q$ in $\bdy \Delta^1$, as in Figure \ref{f:adjacentneighbors}. If the disk $C=E_1 \cup Q \cup E_2$ is a compressing disk for $H$, then we say $C$ is {\it defined} by $E_1$ and $E_2$. 

\begin{figure}
\psfrag{E}{$E_1$}
\psfrag{e}{$E_2$}
\psfrag{f}{$E$}
\psfrag{Q}{$Q$}
\includegraphics[width=4.75in]{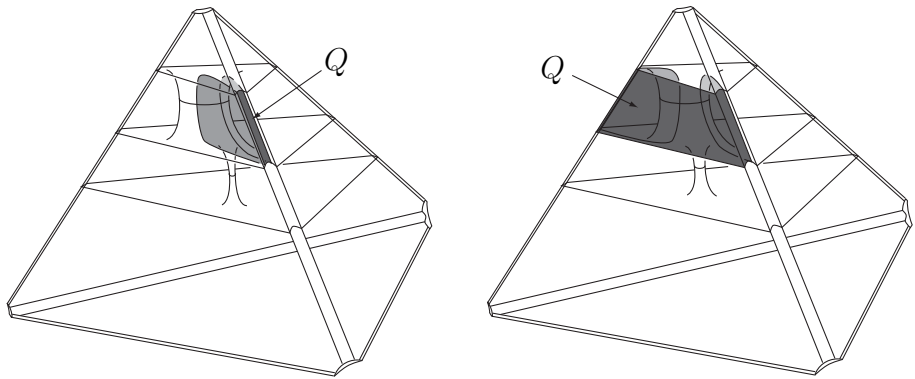}
\caption{The rectangle $Q \subset \bdy \Delta^1$ when $E_1$ and $E_2$ are adjacent neighbors. In the figure on the left, $E_1$ and $E_2$ are incident to the same edge of $\Delta$. In the figure on the right, they are incident to different edges.}
\label{f:adjacentneighbors}
\end{figure}

When $E_1$ and $E_2$ are non-adjacent neighbors, then the arcs $E_1 \cap \bdy \Delta^1$ and $E_2 \cap \bdy \Delta^1$, together with three subarcs of $\bdy H^1$ and a subarc of $\bdy \Delta_v$, bounds a hexagonal region $Q$ of $\bdy \Delta^1$, as in Figure \ref{f:nonadjacentneighbors}. In this case, the disk $C=E_1 \cup Q \cup E_2$ is an edge-compressing disk for $H$, that we say is {\it defined} by $E_1$ and $E_2$. 

\begin{figure}
\psfrag{E}{$E_1$}
\psfrag{e}{$E_2$}
\psfrag{f}{$E$}
\psfrag{Q}{$Q$}
\includegraphics[width=3in]{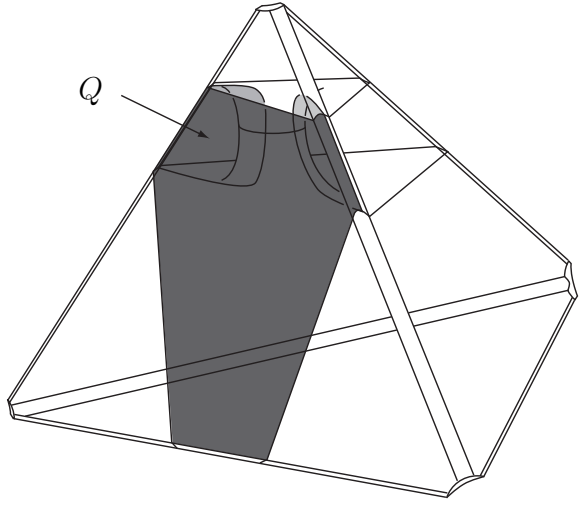}
\caption{The hexagon $Q \subset \bdy \Delta^1$ when $E_1$ and $E_2$ are non-adjacent neighbors.}
\label{f:nonadjacentneighbors}
\end{figure}

We now let $\CC_1$ denote the set of all compressing disks for $H$ defined by pairs of adjacent neighbors in $\EE_0$, together with the set of all edge-compressing disks defined by non-adjacent neighbors in $\EE_0$. Let $\EE_1$ be a set of edge-compressing disks such that
	\begin{itemize}
		\item for each $E \in \EE_1$ there is a $C \in \CC_1$ such that $E \cap C \ne \emptyset$, and
		\item the disks of $\EE_0 \cup \EE_1$ are represented by the vertices of a simplex in $S$.
	\end{itemize}

For example, Figure \ref{f:E1} depicts a surface $H$ and a set $\EE_0$ with three disks. Three pairs of these disks are consecutive, two of those pairs are neighbors, and one of those neighbors defines the compressing disk for the lower tube in the figure, which would then be an element of $\CC_1$. Also depicted in the figure is a single edge-compressing disk that meets this compressing disk, which is hence an element of $\EE_1$. Note that all the disks in $\EE_0 \cup \EE_1$ are disjoint, and so are represented by the vertices of a simplex in $[\CC \EE(H)]$. 

\begin{figure}
\includegraphics[width=3in]{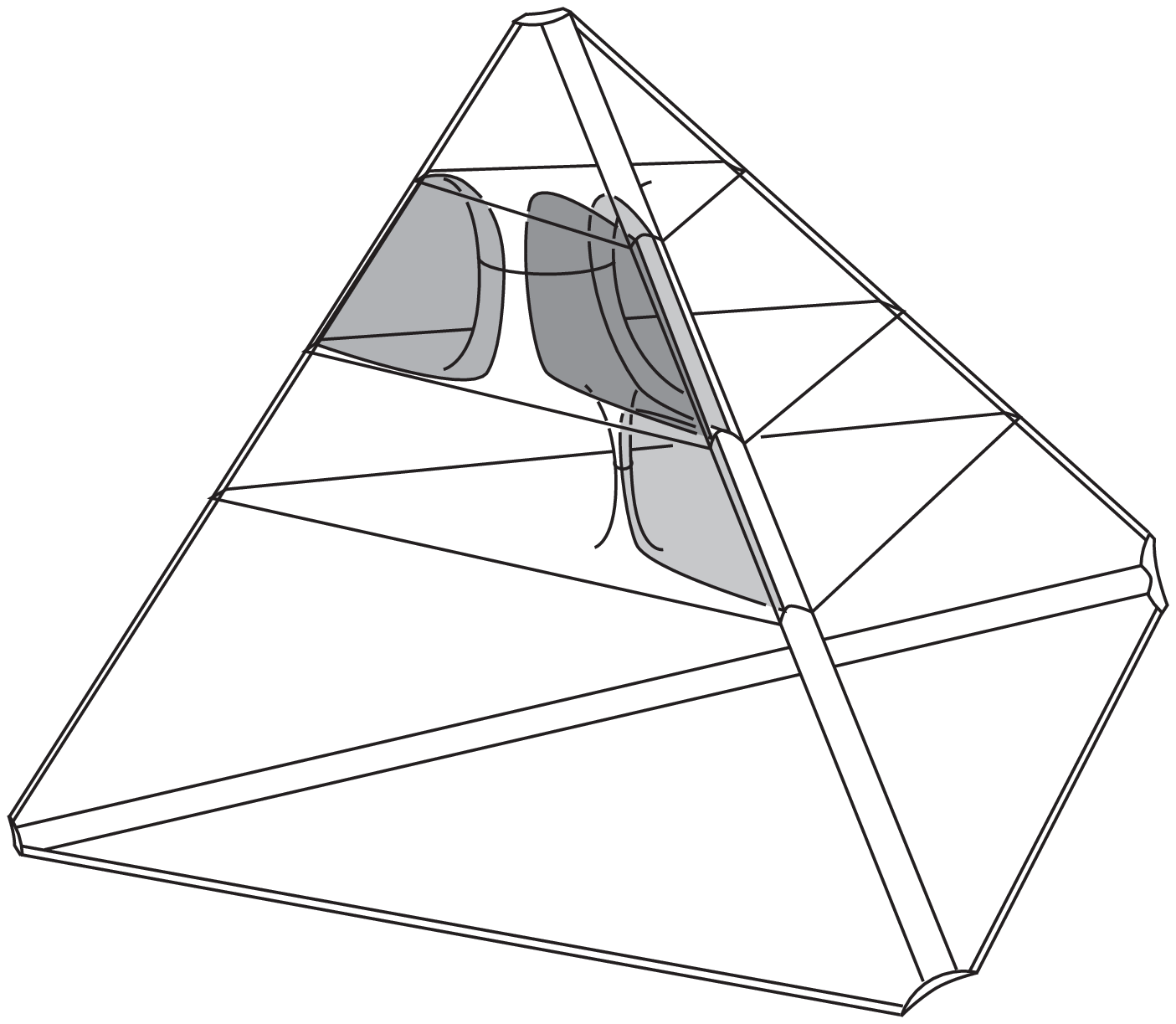}
\caption{Three disks in $\EE_0$ and one in $\EE_1$.}
\label{f:E1}
\end{figure}

We now proceed in the same fashion. Two disks of $\EE_1$ are consecutive if they meet the same disk in $\CC_1$, and there are no other disks of $\EE_1$ between them. The terms {\it neighbor}, {\it adjacent} and {\it non-adjacent} are defined as before, as are compressing disks defined by adjacent neighbors, and edge-compressing disks defined by non-adjacent neighbors. In this way we construct sequences of sets $\{\CC_0, \CC_1, \CC_2, ...\}$ and $\{\EE_0, \EE_1, \EE_2, ...\}$, such that for each $i \ge 1$,

\begin{itemize}
	\item the set $\CC_i$ consists of all compressing disks defined by adjacent neighbors in $\EE_{i-1}$, and edge-compressing disks defined by non-adjacent neighbors in $\EE_{i-1}$, 
	\item the set $\EE_i$ consists of edge-compressing disks such that for each $E \in \EE_i$ there is a $C \in \CC_i$ where $E \cap C \ne \emptyset$, and
	\item the disks of $\bigcup \limits _{j=0} ^i \EE_j$ are represented by the vertices of a simplex in $S$.
\end{itemize}

\begin{dfn}
A simplex $\sigma \subset S$ whose vertices represent the disks $\bigcup \EE_i$ as defined above is called a {\it sunburst of $S$}.
\end{dfn}

\begin{dfn}
Suppose $\sigma$ is a sunburst and $\{\EE_i\}_{i=0}^n$ is the sequence of sets of disks represented by its vertices. Then the {\it sunburst complexity} of $\sigma$ is the $(n+1)$-tuple
\[c(\sigma)=(|\EE_0|,|\EE_1|, ..., |\EE_n|).\]
Two such $(n+1)$-tuples are compared lexicographically.
\end{dfn}

\begin{dfn}
When $\sigma$ is a sunburst that is not a face of any other sunburst, then we say it is {\it full}.
\end{dfn}

\begin{dfn}
The {\it sunburst complexity} of $S$ is defined to be the set of $(n+1)$-tuples
\[\{c(\sigma)|\sigma \mbox{ is a full sunburst of $S$} \}.\] 
We order this set in non-increasing order, and compare two such sets lexicographically. We say an essential sphere $S$ is {\it sunburst minimal} if there is no sphere $S'$ homotopic to $S$ that has lower sunburst complexity. 
\end{dfn}

\begin{lem}
\label{l:NoEmptyTerminals}
Suppose $S$ is sunburst minimal, $\sigma$ is a full sunburst of $S$, $\{\EE_i\}$ is the sequence of sets of disks represented by its vertices, and for each $i$,  $\CC_i$ is the set of compressing and edge-compressing disks defined by neighboring disks in $\EE_i$. Then for each $i \ge 1$ and each disk $C \in \CC_i$, there is a disk $E \in \EE_i$ that meets $C$. 
\end{lem}

\begin{proof}
Let $\Sigma$ be the union of simplices in $S$ that contain $\sigma$. A vertex $[F]$ in $\Sigma-\sigma$ represents a disk $F$ that is disjoint from every element of $\bigcup \EE_i$. Suppose there is an $i$ such that $\CC_i$ contains a disk $C$ that is disjoint from $F$. Then adding $F$ to $\EE_i$ defines a new sunburst which is the join of $\sigma$ and $[F]$, contradicting our assumption that $\sigma$ was full. We conclude that each disk $C \in \CC_i$ is disjoint from every disk represented by a vertex of $\Sigma-\sigma$. 

Now suppose $C$ is a disk in $\CC_i$ that does not meet any disk in $\EE_i$. Such a disk $C$ will then be disjoint from every element of $\bigcup \EE_i$, and thus $C$ is disjoint from every disk represented by a vertex of $\Sigma$. We may thus obtain a sphere homotopic to $S$ by replacing $\Sigma$ with the cone of $\bdy \Sigma$ to the point $[C]$, $\bdy \Sigma \ast [C]$. Since $C$ is a compressing disk, any full sunburst in $\bdy \Sigma \ast [C]$ will have fewer edge-compressing disks, and thus a smaller sunburst complexity. 
\end{proof}

\begin{lem}
\label{l:adjacent}
Suppose $S$ is sunburst minimal,  $\sigma$ is a full sunburst of $S$, and $\{\EE_i\}$ is the sequence of sets of disks represented by its vertices. Then for each $i$, all neighbors in $\EE_i$ are adjacent. 
\end{lem}

\begin{proof}
Suppose $E_1$ and $E_2$ are non-adjacent neighbors in $\EE_i$, and let $C$ denote the edge-compressing disk in $\CC_{i+1}$ that they define. By Lemma \ref{l:NoEmptyTerminals}, $C$ meets some disk $E \in \EE_{i+1}$. Since $E_1$ and $E_2$ are non-adjacent, they must be incident to different edges of $\Delta$. The disk $E$ will then be incident to the same edge as $E_1$, or the same edge as $E_2$. Assume the former, and let $e$ denote the edge of $\Delta$ that both $E_1$ and $E$ meet. 

\begin{figure}
\psfrag{Q}{$Q$}
\psfrag{E}{$E_1$}
\psfrag{e}{$E$}
\[\includegraphics[width=3in]{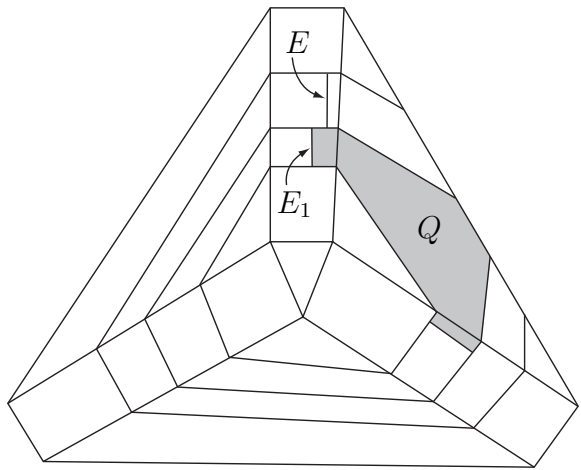}\]
\caption{The disk $E$ is further right than the disk $E_1$, which meets the right switch.}
\label{f:NoNonAdjacent}
\end{figure}

Comparing Figures \ref{f:Switches} and \ref{f:nonadjacentneighbors}, we see that $E_1$ meets $R_e$ (the frontier of $N(e)$) in either a right or left switch. Note that since $E$ is disjoint from $E_1$ and $E_2$, then it must meet the disk $Q$, where $C=E_1 \cup Q \cup E_2$ and $Q$ is the shaded region of $\Delta^1$ depicted in Figure \ref{f:nonadjacentneighbors}. The three simplices of $\bdy \Delta^1$ that are incident to the vertex $v$ are depicted in Figure \ref{f:NoNonAdjacent}, with $Q$ depicted again as a shaded region. From the figure, we see that the disk $E$ is to the right of $E_1$ on $e$. Hence, the simplex  $\sigma$ will have a rightmost disk $E'$ on $e$ which has a push-off $E^*$ to its right, which is an edge-compressing disk incident to an edge that does not meet the vertex $v$. By construction then, $E^*$ will be disjoint from every disk represented by a vertex of $\sigma$. 

Let $\Sigma$ be the union of simplices in $S$ that contain $\sigma$. Let $[F]$ be a vertex of $\Sigma -\sigma$. As $E^*$ is a push-off, it consists of the disk $E'$, together with a rectangular subregion of $\bdy \Delta^1$, as depicted in Figure \ref{f:PushOff}. Hence, if $F$ meets $E^*$ it must also meet $E'$. However, $[F]$ is by definition connected by an edge to each vertex of $\sigma$, and thus $F$ is disjoint from every disk represented by a vertex of $\sigma$. In particular, $F$ is disjoint from $E'$. We conclude $F$ must also be disjoint from $E^*$. 

We conclude $E^*$ is disjoint from every disk represented by a vertex of $\Sigma$. We may thus obtain a sphere homotopic to $S$ by replacing $\Sigma$ with the cone of $\bdy \Sigma$ to the point $[E^*]$. Every full sunburst simplex of such a cone is spanned a proper subset of the vertices of $\sigma$. Hence, any resulting full sunburst will have lower sunburst complexity. 
\end{proof}

\begin{thm}
\label{t:TriangleCompress}
Suppose $S$ is sunburst minimal, $\sigma$ is a full sunburst of $S$, $\{\EE_i\}$ is the sequence of sets of disks represented by its vertices, and for each $i$,  $\CC_i$ is the set of compressing disks defined by neighboring disks in $\EE_i$.  Then the local index of $H/ \bigcup \CC_i$ is at most $n-\left| \bigcup \EE_i \right|$.
\end{thm}

\begin{proof}
Let $\Sigma$ be the union of simplices in $S$ that contain $\sigma$. We claim $\Sigma - \sigma$ is an essential sphere in $[\CC \EE(H/ \bigcup \CC_i)]$. To establish this, there are a number of things we will have to prove. 

\begin{clm}
Each vertex in $\Sigma - \sigma$ represents a disk in $\CC \EE(H/ \bigcup \CC_i)$. 
\end{clm}

\begin{proof}
Suppose $[D]$ is such a vertex, representing a disk, $D$. We first deal with the case that $D$ is an edge-compressing disk. If such a disk is disjoint from every disk in $\CC_i$, for each $i$, then it will be an edge-compressing disk for $H/\bigcup \CC_i$. However, if $D$ meets some disk $C \in \CC_i$, for some $i$, then $D$ would have belonged to the set $\EE_{i+1}$, and thus $[D]$ would have been a vertex of $\sigma$ (since $\sigma$ is assumed to be full). 

We now deal with the case that $D$ is a compressing disk. We again claim that it must be disjoint from every disk in $\CC_i$, for each $i$. Suppose $C \in \CC_i$ is a disk such that $D \cap C \ne \emptyset$. Recall that $C$ is either the disk $C_0$, or by Lemma \ref{l:adjacent}, $C=E_1 \cup Q \cup E_2$ for some rectangle $Q \subset \Delta^1$ and adjacent disks $E_1 , E_2 \in \EE_{i-1}$. In the former case, the disk $C$ is isotopic into $\bdy \Delta^1$, and is therefore disjoint from every compressing disk such as the disk $D$. In the latter case, we know $D$ is disjoint from $E_1$ and $E_2$, since $[D]$ is a vertex in the link of $\sigma$, and $[E_1]$ and $[E_2]$ are vertices of $\sigma$. Furthermore, $D$ is disjoint from $Q$, as $Q \subset \bdy \Delta^1$. Therefore $D$ is disjoint from $C=E_1 \cup Q \cup E_2$.

We have established that $D$ is disjoint from every disk in $\CC_i$, for all $i$. To show $D$ is a compressing disk for $H/\bigcup \CC_i$, we must now show that $\bdy D$ is essential on $H/\bigcup \CC_i$. If not, then $\bdy D$ bounds a subdisk $D'$ of $H/\bigcup \CC_i$ that contains scars of a disk $C \in \CC_i$, for some $i$.  (See Definition \ref{d:H/D}.) By Lemma \ref{l:NoEmptyTerminals}, there is an edge-compressing disk $E \in \EE_i$ for $H$ that meets $C$. Thus, $E \cap \bdy D \ne \emptyset$, and hence $E \cap D \ne \emptyset$. We have reached a contradiction, as $D$ is disjoint from every disk in $\EE_i$, for all $i$. 
\end{proof}

The previous claim implies that there is a natural map $\phi$ from $S'=\Sigma - \sigma$ into $[\CC \EE(H/ \bigcup \CC_i)]$. We must now show that this map is homotopically non-trivial. If not, then we can extend $\phi$ to a map from a ball $B$ bounded by $S'$ into $[\CC \EE(H/ \bigcup \CC_i)]$. We assume $B$ has been triangulated so that this map is simplicial. 

We now construct a map $\eta:B \to [\CC \EE(H)]$ such that 
\begin{enumerate}
	\item $\eta \circ \phi$ fixes $S'$, and 
	\item each disk represented by a vertex of $\eta(B)$ is disjoint from every disk represented by a vertex of $\sigma$. 
\end{enumerate}

To prove the existence of $\eta$ we must show that for each disk $D \in \CC \EE(H/ \bigcup \CC_i)$ there is a disk $\tilde D \in \CC \EE(H)$, disjoint from every disk represented by a vertex of $\sigma$, that becomes $D$ after compressing. It will follow from this construction that if $D$ and $D'$ are disjoint disks in $\CC \EE(H/ \bigcup \CC_i)$ then $\tilde D$ and $\tilde D'$ can be chosen to be disjoint. Thus, the map $\eta$ that send $D$ to $\tilde D$ will map $B$ simplicially into $[\CC \EE(H)]$.

\begin{figure}
\includegraphics[width=4in]{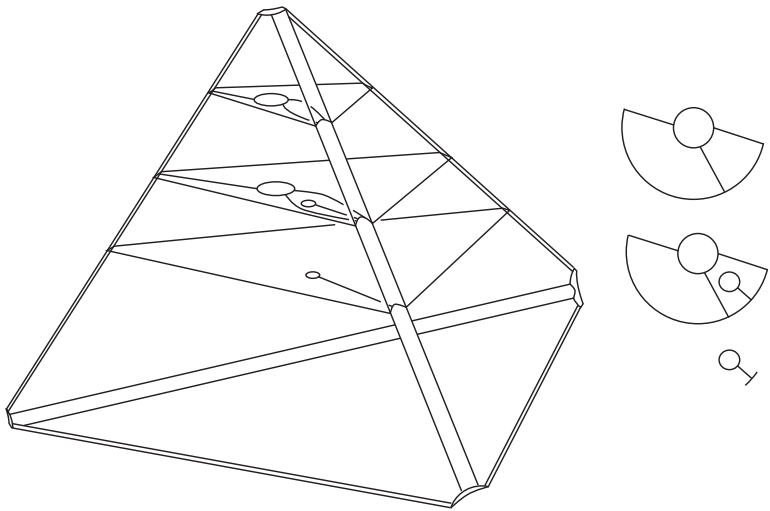}
\caption{The surface $H/\bigcup \CC_i$, when $H$ is the surface of Figure \ref{f:E1}.}
\label{f:Exploded}
\end{figure}

The map $\eta$ will be easily defined once we better understand the relationship between the surfaces $H$ and $H/\bigcup \CC_i$. When $H$ is the surface depicted in Figure \ref{f:E1}, the surface $H/\bigcup \CC_i$ is depicted on the left in Figure \ref{f:Exploded}. The small circles in this figure bound disks that are the scars of the disks in $\CC_0 \cup \CC_1$. The arcs that meets these disks represent the places where the disks of $\EE_0 \cup \EE_1$ met $H$. On the righthand side of this figure we depict the following: The circle at the top depicts the upper scar of $C_0$. The three arcs incident to it represent places where the three disks of $\EE_0$ met $H$. Two pairs of these arcs represent neighboring disks in $\EE_0$, and are thus connected by subarcs of $\bdy H$. Just below this we depict the other scar of $C_0$. Again there are three arcs incident to this scar, also representing places where the three disks of $\EE_0$ met $H$. As in the upper figure, for each pair of neighboring disks in $\EE_0$ there is a subarc of $\bdy H$ depicted joining them. Furthermore, one such pair defines a compressing disk $C_1$ in $\CC_1$. A scar of $C_1$ is depicted by the smaller circle of the middle figure, together with an arc that represents where a disk of $\EE_1$ met $H$. Finally, the lowest figure depicts the other scar of $C_1$, together with the other arc where a disk of $\EE_1$ met $H$.

We call the entire righthand figure a {\it sunburst diagram.} Formally, a sunburst diagram is a depiction of each scar of a disk of $\bigcup \CC_i$, together with arcs representing all of the places where disks of $\bigcup \EE_i$ met $H$, along with subarcs of $\bdy H$ connecting arcs representing neighboring disks of $\bigcup \EE_i$. Note that between any pair of neighboring disks such a subarc of $\bdy H$ must exist, because otherwise such a pair would be non-adjacent. By Lemma \ref{l:adjacent}, there are no such pairs.

\begin{figure}
\includegraphics[width=4in]{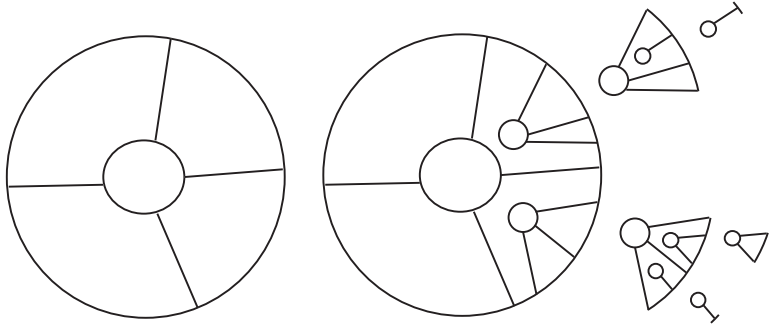}
\caption{A more complicated sunburst diagram.}
\label{f:Diagram}
\end{figure}

Sunburst diagrams can be far more complicated than the one depicted in Figure \ref{f:Exploded}. Knowing what features such a diagram must necessarily possess will be helpful in understanding the rest of the proof. As an illustrative example, consider Figure \ref{f:Diagram}. Such a diagram comes from a situation in which there are four disks in $\EE_0$. At least one of these disks meets each edge of $\Delta$ that is incident to $v$, and hence every pair of consecutive disks are neighbors. By Lemma \ref{l:adjacent}, such neighbors must be adjacent, so there is an arc of $\bdy H$ connecting them. This is why the boundaries of the first and second figure in the diagram are complete circles. In the figure, two neighboring pairs of $\EE_0$ define two compressing disks in $\CC_1$. Each of the disks in $\CC_1$ meets three disks in $\EE_1$. Three pairs of the disks in $\EE_1$ define three compressing disks in $\CC_2$. Two of these meet a single disk in $\EE_2$, and one meets two disks in $\EE_2$. Finally, the only two neighboring disks in $\EE_2$ are parallel on $H$, and therefore do not define any further compressing disks. 

\begin{clm}
\label{c:CircleImpliesTriangle}
The subset of $\bdy H$ represented in a component of a sunburst diagram is connected, and is thus either an arc or a circle. In the former case, the remnants of $\bigcup \EE_i$ depicted in this component are incident to at most two edges of $\Delta$. In the latter case, the depicted circle of $\bdy H$ bounds a component of $H/\bigcup \CC_i$ that is a normal triangle. 
\end{clm}

\begin{proof}
Each arc in a sunburst diagram represents a disk of $\EE_i$ that is incident to one of the three edges of $\Delta$ that meet $v$. If all three edges of $\Delta$ are met, then each pair of consecutive disks of $\EE_i$ represented in the diagram are neighbors. By Lemma \ref{l:adjacent}, such a pair is adjacent, and thus they are connected by an arc of $\bdy H$. As there are three edges of $\Delta$ incident to $v$, these arcs piece together to form a length three normal loop on $\bdy \Delta$. After compressing along the disks of $\bigcup \CC_i$, this loop will bound a normal disk. 

If at most two of the edges of $\Delta$ are met by disks of $\EE_i$ represented in a component of a sunburst diagram, then exactly one pair of consecutive disks of $\EE_i$ are not neighbors. Hence, all of the other pairs are joined by arcs of $\bdy H$ to form a connected arc. 
\end{proof}

We now complete the proof of Theorem \ref{t:TriangleCompress}. Suppose first $D$ is a compressing or edge-compressing disk for $H/\bigcup \CC_i$. As each of the scars of the disks in $\bigcup \CC_i$ is a disk, $\bdy D$ can be isotoped to be disjoint from these. If $\bdy D$ meets any of the remnants of the disks of $\bigcup \EE_i$, then $\bdy D$ will be depicted in some component of a sunburst diagram. The subset of $\bdy H$ depicted in this component can not be a circle, since Claim \ref{c:CircleImpliesTriangle} would then imply that this component represents a normal triangle of $H/\bigcup \CC_i$. However, no compressing or edge-compressing disk for $H/\bigcup \CC_i$ will meet a normal triangle piece of this surface. 

\begin{figure}
\psfrag{D}{$\bdy D$}
\[\includegraphics[width=4in]{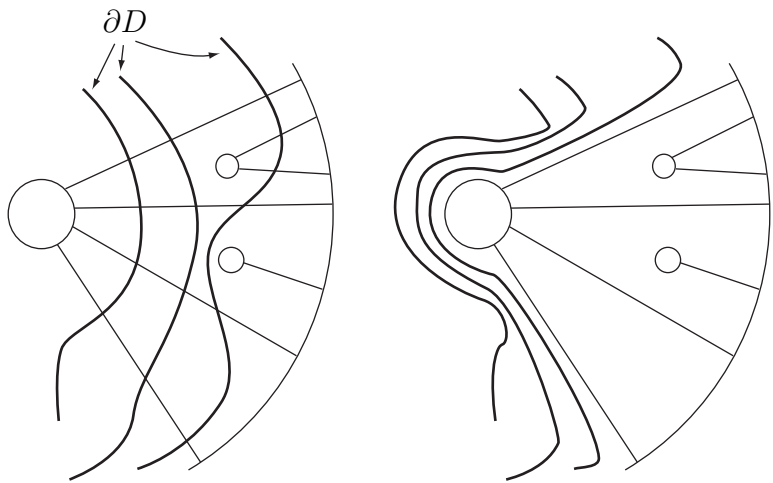}\]
\caption{On the left is a component of a sunburst diagram that meets a compressing disk $D$. Since the depicted subset of $\bdy H$ is a connected arc, $\bdy D$ can be isotoped to go around this component of the sunburst diagram, as in the figure on the right.}
\label{f:DIsotopyCompressing}
\end{figure}

Claim \ref{c:CircleImpliesTriangle} now implies that the subset of $\bdy H$ depicted in any component of the sunburst diagram that meets $\bdy D$ is an arc. Hence, this component looks something like the one depicted in Figure \ref{f:DIsotopyCompressing}. The salient feature of this figure is that there is a (larger) circle representing a scar of some disk $C$ of $\CC_i$ (for some $i$), arcs of $\EE_i$ connecting this circle to a connected subarc of $\bdy H$, and that the entire figure is homeomorphic to a disk. Thus, any intersection of $\bdy D$ with this component can be isotoped off. When $D$ is a compressing disk, such an isotopy is depicted in Figure \ref{f:DIsotopyCompressing}. When $D$ is an edge-compressing disk some additional care must be taken.  The depicted disks of $\EE_i$ are $\{E_j\}_{j=1}^m$, where for each $j$, the disks $E_j$ and $E_{j+1}$ are adjacent neighbors. These disks are all incident to at most two edges of $\Delta$. Hence, we may assume that for some $k \le m$, the disks $\{E_j\}_{j=1}^k$ are all incident to the same edge of $\Delta$. The disk $D$ is then either incident to the same edge as these disks, or $k < m$ and $D$ is incident to the same edge as the disks $\{E_j\}_{j=k+1}^m$. Assume the former. Then $\bdy D$ can be isotoped to be disjoint from every disk in $\{E_j\}_{j=1}^m$ via an isotopy as depicted in Figure \ref{f:DIsotopyEdgeCompressing}.

\begin{figure}
\psfrag{E}{$E_1$}
\psfrag{D}{$\bdy D$}
\[\includegraphics[width=4in]{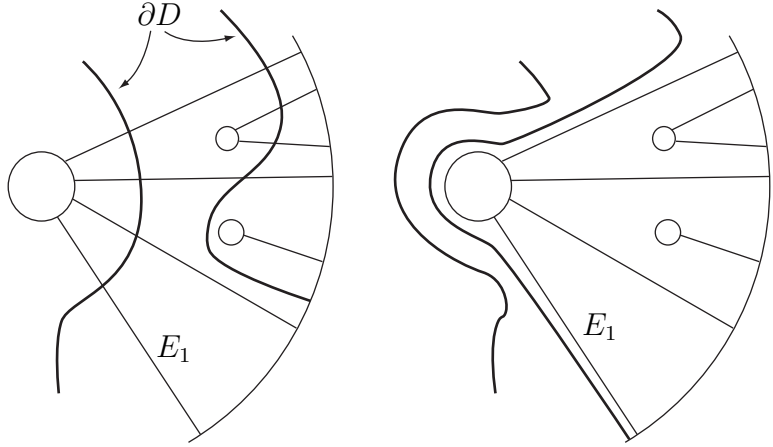}\]
\caption{A component of a sunburst diagram that meets an edge-compressing disk $D$. The disk $E_1$ is incident to the same edge of $\Delta$ as $D$.}
\label{f:DIsotopyEdgeCompressing}
\end{figure}

Once $\bdy D$ has been isotoped to be disjoint from every component of a sunburst diagram, it follows from a standard innermost disk argument that the disk $D$ can be made disjoint from every disk of $\bigcup \CC_i$ and $\bigcup \EE_i$. Thus, $D$ is now the desired compressing or edge-compressing disk $\tilde D$ for $H/\bigcup \CC_i$, and is disjoint from every disk represented by a vertex of $\sigma$. 

\begin{figure}
\psfrag{s}{$\sigma$}
\psfrag{S}{$\Sigma$}
\psfrag{B}{$\eta(B)$}
\[\includegraphics[width=4in]{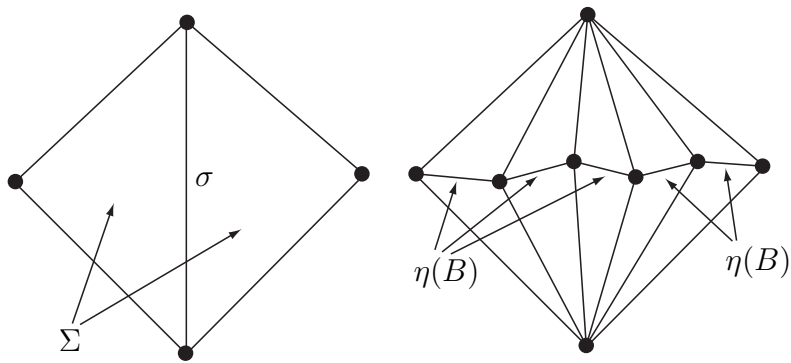}\]
\caption{In the left is the simplex $\sigma$, and the simplices $\Sigma$ containing $\sigma$. On the right is the join of $\eta(B)$ with $\bdy \sigma$.}
\label{f:eta}
\end{figure}

The proof is now complete by observing that the existence of the ball $B$ and map $\eta$ implies that we can modify $S$ by replacing $\Sigma$ with the join of $\eta(B)$ and $\bdy \sigma$, as in Figure \ref{f:eta}.  This will lower the sunburst complexity of $S$, providing the desired contradiction. 
\end{proof}

The following lemma completes the proof of the third and fourth conclusion of Theorem \ref{t:summary}.

\begin{lem}
Suppose $H$ is a normally connected surface in a tetrahedron with local index $n$. Then $H$ is connected, and $|\bdy H| \le 4(n+1)$. 
\end{lem}

\begin{proof}
First, suppose every component of $\bdy H$ has length three or four. By Lemma \ref{l:NoLongConnected}, $H$ is then connected. By the second conclusion of Theorem \ref{t:summary}, $1-\chi(H) \le n$. It follows that $H$ has at most $n+1$ boundary components. By assumption, each boundary component has length at most four, so we have $|\bdy H| \le 4(n+1)$. 

When every component of $\bdy H$ has length at least eight, then $H$ is connected by Lemma \ref{l:OnlyLongConnected} and $|\bdy H| \le 4(n+1)$ by Theorem \ref{t:NoLong}. 

We are left with the case that $\bdy H$ has some components of length at least 8, and some components of length 3. It follows from Lemma \ref{l:SameComponent} that all parallel length 3 components are on the same component of $H$, and each such parallel family is on the same component of $H$ as at least one longer component of $\bdy H$. We will show that all of these longer components of $\bdy H$ are on the same component of $H$, and thus conclude that $H$ must be connected. 

Let $\{\EE_i\}$ and $\{\CC_i\}$ denote the collections of disks of Theorem \ref{t:TriangleCompress}. By Lemma \ref{l:NoEmptyTerminals}, $|\bigcup \EE_i| \ge |\bigcup \CC_i|$. By Theorem \ref{t:TriangleCompress}, the local index of $H/ \bigcup \CC_i$ is at most $n-\left| \bigcup \EE_i \right|$, and thus at most $n-\left| \bigcup \CC_i \right|$. By construction, the set $\CC_0$ contains a compressing disk $C_0$ such that $H/C_0$ contains a component that is a normal triangle. Thus, $H/ \bigcup \CC_i$ contains at least one normal triangle component. Furthermore, the total number $t$ of normal triangle components of $H/ \bigcup \CC_i$ is at most $|\bigcup \CC_i|$. Thus, the local index of $H/ \bigcup \CC_i$ is at most $n-t$. 

We now continue to apply Theorem \ref{t:TriangleCompress} to obtain a sequence of surfaces. Each such surface is obtained from the previous one by some compressions, followed by discarding any resulting normal triangle components. This process concludes with a surface $H'$ with only boundary curves of length at least eight. By the preceding paragraph, the local index $n'$ of $H'$ will be at most $n$ minus the total number of normal triangles lost, which is precisely the number $m$ of length three components of $\bdy H$. That is, $n' \le n-m$. As above, it follows from Lemma \ref{l:OnlyLongConnected} that $H'$ is connected, and from Theorem \ref{t:NoLong} that $|\bdy H'| \le 4(n'+1)$. Thus, $H$ was connected, and 
\begin{eqnarray*}
|\bdy H| &=&|\bdy H'|+3m\\ 
&& \le 4(n'+1) +3m\\ 
&& \le 4(n-m+1)+3m\\
&& \le 4(n+1).
\end{eqnarray*}
\end{proof}

\section{Questions and Conjectures}

Theorem \ref{t:summary} narrows down the possible pictures for a surface with a prescribed local index in a tetrahedron to a finite list. However, as noted in the introduction, this list will almost certainly contain surfaces that do not have the desired local index. Resolving the following questions would help give an exact classification. 

For all of the following questions, $H$ will denote a connected surface with local index $n$ in a tetrahedron $\Delta$. 

\begin{quest}
Can $\bdy H$ have more than one ``long" component?
\end{quest}

Here, a ``long component" is one that meets the 1-skeleton at least eight times. In particular, it does not seem likely that two parallel $n$-gons, connected by a single unknotted tube, will be locally topologically minimal. 

\begin{quest}
Does every long component of $\bdy H$ have to be the boundary of a helicoid?
\end{quest}

By \cite{TopMinNormalII}, if $H$ is a disk, then it is a helicoid. This rules out, for example, certain loops of length 20 from being the boundary of a locally topologically minimal disk. Can such loops arise as boundary components of locally topologically minimal surfaces that are not disks?

A more careful analysis of the proof of Theorem \ref{t:NoLong} reveals that when the boundary of $H$ is not helical, then the bound given in that theorem can be improved to $|\bdy H| \le 2n+4$. As the shortest possible non-helical loop on $\bdy \Delta$ has length 20, we conclude that the lowest possible local index of a surface with non-helical boundary is 8. 

\begin{quest}
If $\bdy H$ contains a triangular component then does the surface obtained by compressing $H$ along an outermost one reduce the index of $H$ by exactly one?
\end{quest}

Theorem \ref{t:TriangleCompress} implies that in this case there will be a collection of compressing disks that will pinch off this triangle, and lower the index by some amount. Can one improve the proof so that this collection always has exactly one element?

\begin{quest}
Can one determine the local index of $H$, just from knowing its Euler characteristic and total boundary length?
\end{quest}

We would like to think of local index as a commodity. Theorem \ref{t:summary} says that you can pay for this commodity with either negative Euler characteristic or boundary length. This question asks whether or not there is an exchange rate between these two currencies.

\bibliographystyle{alpha}

\end{document}